\documentclass[11pt,letterpaper]{amsart}
\usepackage{fullpage}
\usepackage{amsmath,amssymb,amsthm,enumerate,bbm,color}

\newtheorem{bigthm}{Theorem}
\newtheorem{thm}{Theorem}[section]
\newtheorem{lem}[thm]{Lemma}
\newtheorem{cor}[thm]{Corollary}

\theoremstyle{definition}

\theoremstyle{definition}

\theoremstyle{definition}

\theoremstyle{definition}

\theoremstyle{definition}
\newtheorem{defn}[thm]{Definition}
\theoremstyle{definition}

\theoremstyle{definition}

\theoremstyle{definition}
\newtheorem{rem}[thm]{Remark}
\theoremstyle{definition}

\newcommand{\symdif}{\vartriangle}

\numberwithin{equation}{section}

\def\ind{{\bf 1}}
\def\Pa{\vspace{11pt}\noindent}
\def\R{\mathbb{R}}
\def\Z{\mathbb{Z}}

\def\B{\mathcal{B}}
\def\to{\rightarrow}

\title{Perturbation of Sparse Ergodic Averages}
\author{Andrew Parrish}

\begin{document}

\begin{abstract}
 We provide examples of a nested sequences of sets $\{ S_n \}$, suitably sparse, residing in a group $G$, for which the averages 
\begin{equation*}
 A(N,f) = \frac{1}{\#S_N}  \sum_{g \in S_N} f\left( T_g x \right)
\end{equation*}
fail converge pointwise for $f$ in certain $L^p$ spaces, but do converge in others, for any free group action $T$. Our construction involves the method of perturbation pioneered by A. Bellow and applied in the integer cases by K. Reinhold and M. Wierdl. 
\end{abstract}

\maketitle
\tableofcontents 

\section{Introduction}
Let $d$ be a natural number. In this paper we seek to show that there are sequences of sets $\left\{S_t \right\}$ in $\Z^d$ for which the averages 
\begin{equation*}
   A\left(S_t, f \right) = \frac{1}{\# S_t} \sum_{g \in S_t} f\left(T_gx \right)
\end{equation*}
converge pointwise for a.e. $x \in X$ for any free $\Z^d$-action, $(X, \B, m, T)$, and for any $f \in L^q$, but diverge for some $f$ in a larger Banach space, again for any free $\Z^d$-action.

\Pa
All measure-preserving systems (group actions) mentioned should be understood to be aperiodic (free); all measure spaces are of finite measure. $\#A$ denotes the cardinality of the discrete set $A$. 

\Pa
One of the interesting aspects of working in higher dimensions is that it becomes necessary to define in what order addition is to be performed. As it turns out, this is not a trivial process- adding elements as they appear in successive balls results in different convergence properties for subsequences compared to other nested F\o lner sequences. 

\Pa
We will be interested in only two sequences  in this regard: the sequence of balls of radius $N$, $\{B_N\}$, centered at the origin, and the sequence of cubes, $\{R_N\}$, likewise centered at the origin and of side length $2N+1$. While most of our proofs will remain true for any sequence of nested rectangular prisms in $\Z^d$, we adopt the cubes for the sake of simplicity.

\Pa
\begin{defn}
   Suppose $A$ is a set in $\Z^d$. Then 
\begin{equation*}
   A(N) := A \cap B_N \mbox{, and } A_N := A \cap R_N.
\end{equation*}
\end{defn}

\Pa
\begin{defn}
Suppose that $(X, \B, m, T)$ is a measure preserving system and  $p \geq 1$. We say that a sequence $\{a_n\}$ is \textit{pointwise $L^p$-good for $(X, \B, m, T)$}, if the sequence
\begin{equation*}
   A_N\left(a_n, f \right) = \frac{1}{N} \sum_{n=1}^N f\left(T^{a_n} x \right)
\end{equation*}
converges for a.e. $x \in X$ for all $f \in L^p$. A sequence is \textit{universally pointwise $L^p$-good}  if it is pointwise $L^p$ good for every aperiodic measure-preserving system. 

\Pa
Sequences that fail to be $L^p$-good for $(X, \B, m, T)$ we call \textit{ pointwise $L^p$-bad for $(X, \B, m,T)$}. A sequence that is pointwise $L^p$ bad for every aperiodic measure-preserving system is \textit{universally pointwise $L^p$-bad}.
\end{defn}

\Pa
\begin{defn}
We say that a sequence is \textit{universally $\infty$-sweeping out} for $L^p$, $p \geq 1$ if there is an $f \in L^p$ such that
\begin{equation*}
   \sup_N  A_N\left(a_n, f \right) = \infty.
\end{equation*}
\end{defn}
\Pa
The most familiar example of an $L^1$-good sequence is the natural numbers themselves. While no positive density sequence can be universally $L^1$-bad, the first example of a zero-density, universally pointwise $L^1$-good sequence was given in \cite{Bellow1984}. In \cite{Bellow1989Perturb}, Bellow constructs a universally $L^p$-good sequence that is universally $L^q$-bad for $1 \leq q < p$ and any $1 < p < \infty$.  Using similar methods, Reinhold (\cite{Reinhold1994}) showed that there is a sequence which is universally $L^p$-good for $p > q \geq 1$ but universally $L^q$-bad and constructed sequences which are $L^q$-bad for all $q < \infty$ but good in $L^\infty$. 

\Pa
The method employed in these last two results was perturbation- by changing slightly the sequence, the averages were forced to converge for some Banach spaces but not others. These changes provided additional structure to the sequences. This structure was then shown to allow divergence for certain functions. Further, albeit circumstantial, evidence for the link between structure and divergence can be found when we consider the recent results of \cite{Buczolich2010}, \cite{ChristPreprint}, and \cite{LPR}; while the sequence of squares, $\{n^2 \}$, is pointwise $L^1$-bad, we see that there are examples of sequences with zero Banach density that grow at the same rate and are universally $L^1$-good.

\Pa
In order to continue the discussion in higher dimensions, we must reframe our definitions regarding density and ``goodness'' and ``badness'' of sequences in terms of sequences of sets in $\Z^d$. 
\begin{defn}
   Suppose $D \subseteq \Z^d$. We say that $D$ is \textit{sparse} if
\begin{equation*}
   \lim_{N \to \infty}\frac{\# \left(D \cap B_N \right)}{\#B_N}=\lim_{N \to \infty} \frac{\# D(N)}{\#B_N} =0,
\end{equation*}
   and \textit{Banach density $0$} if 
\begin{equation*}
   \lim_{N \to \infty} \sup_{g \in \Z^d} \frac{\# \left(g+F \cap B_N \right)}{ \#B_N}=0.
\end{equation*}
\end{defn}

\begin{defn}
Suppose that $(X, \B, m, T)$ is a free $\Z^d$-action and $p \geq 1$. We say that a sequence of sets $\{S_n\}$ is \textit{pointwise $L^p$-good for $(X, \B, m, T)$}, if the sequence
\begin{equation*}
   A\left(S_n, f \right) = \frac{1}{\# S_n} \sum_{g \in S_n} f\left(T_g x \right)
\end{equation*}
converges (as $n \to \infty$) for a.e. $x \in X$ for all $f \in L^p$. A sequence is \textit{universally pointwise $L^p$-good}  if it is pointwise $L^p$ good for every free $\Z^d$-action.

\Pa
A sequence of sets $\{S_n\}$ is \textit{universally $\infty$-sweeping out} for $L^p$, $p \geq 1$ if there is an $f \in L^p$ such that
\begin{equation*}
   \sup_n  A\left(S_n, f \right)  = \infty.
\end{equation*} 
\end{defn}

\Pa
Finally, we also require a definition of perturbation in terms of these sequences.
\begin{defn}
 Let $\{D_n \}$ be a sequence of sets. The sequence $\{S_n \}$ is a perturbation of $\{D_n \}$ if
\begin{equation*}
 \lim_{n \to \infty} \frac{\# \left( S_n \symdif D_n \right)}{\# D_n} =0,
\end{equation*}
\end{defn}

\Pa
We may now state our main results: 
\begin{bigthm}\label{One}
   Let $1 \leq q < \infty$, $D$ be sparse in $\Z^d$, and suppose that $\left\{ D_N \right\}$ is a universally good sequence of sets for $L^p$ for every $p>q$. Then there is a perturbation of $\left\{ D_N \right\}$,  $\left\{ S_N \right\}$, so that $\left\{ S_N \right\}$ is universally good for each $p > q$, but is $\infty$-sweeping out for $L^q$. 
\end{bigthm}

\begin{bigthm}\label{Two}
   Let $1 < q < \infty$ and suppose that $\left\{ D_N \right\}$ is universally good for $f \in L^q$. Then there is a perturbation of $\left\{ D_N \right\}$,  $\left\{ S_N \right\}$, so that  $\left\{ S_N \right\}$ remains universally good for $q$, but is $\infty$-sweeping out for any $p< q$. 
\end{bigthm}

\begin{cor}\label{Balls}
   The same results hold for $\left\{ D(N) \right\}$ and $\left\{ S(N) \right\}$ if there are constants $c, C,$ and $C'$ so that 
\begin{align}
   c < \frac{\# D(N)}{ \#D_N} < C \mbox{, and} \label{comparable}\\
   \# D(2N) < C' \# D(N). \label{growth}
\end{align}
\end{cor}

\section{A Rohlin Transferrence Lemma}
\begin{lem}\label{Rohlin}
   Suppose $S$ is sparse in $\Z^d$ and let $\left\{ R_N \right\}$ be a the sequence of cubes in $\Z^d$ of side length $2N+1$ centered at the origin. Let $S_N = S \cap R_N$. \\
   If, for every positive $K$ and $\varepsilon < 1$ there is an $f: \Z^d \to \R$ and a finite set $\Lambda$ such that 
\begin{align}
   &\limsup_{N \to \infty} \frac{1}{\# R_N} \sum_{n \in R_N} \left| \Phi \left( f(n) \right) \right| \leq 1, \mbox{ and}\\
   &\limsup_{N \to \infty} \frac{\# \left\{ j : \max_{N \in \Lambda} \frac{1}{\# S_N} \sum_{n \in S_N} \left|f (j+n) \right| >K\right\}}{\# R_N} \geq 1- \varepsilon, \label{transassum}
\end{align}
then $\left\{ S_N \right\}$ is a universally $\infty$-sweeping out sequence of sets for $\Phi(L)$.
\end{lem}

\begin{proof}
   The proof of the lemma rests entirely on the Rohlin Lemma for $\Z^d$. For this, we refer the reader to \cite{Conze1972} and \cite{Katznelson1972}. 

  \Pa 
  We will restrict our attention throughout to positive-valued functions. The assumption of the lemma, (\ref{transassum}), is equvalent to 
\begin{equation*}
   \limsup_{t \rightarrow \infty} \frac{\#\left\{ \vec{j} \in R_t : \max_{N \in \Lambda} \frac{1}{\# S_N} \sum_{\vec{n} \in S_N} f\left( \vec{j} + \vec{n} \right) \geq K \right\}}{ \#R_t} \geq 1-\epsilon,
\end{equation*}
which is to say that there are infinitely many $t$ so that 
\begin{equation*}
   \frac{\#\left\{ \vec{j} \in R_t : \max_{N \in \Lambda} \frac{1}{\# S_N} \sum_{\vec{n} \in S_N} f\left( \vec{j} + \vec{n} \right) \geq K \right\}}{ \#R_t} \geq 1-\epsilon.
\end{equation*}
Applying the $\Z^d$ Rohlin Lemma, we will construct a collection of $d$ Rohlin towers, each of height $2t+1$. What we want to do is make sure that our averages don't run off the top of our towers- there needs to be enough space at the top to accommodate every element of $S_N$.  

\Pa
Let $\delta = \max_{\vec{n} \in S_N, 1 \leq i \leq d} \left| n_i \right|$, and define 
\begin{equation*}
   R_{t, \delta} = R_t \setminus \left\{ \vec{n} : |n_i| \in [t- \delta, t], 1\leq i \leq d \right\}.
\end{equation*}
This leaves enough space at the edges of $R_t$ to acommodate $S_N$ and removes only $(2t+1)^d - (2t+1 - 2\delta)^d$ elements. 

\Pa
Fix $0< \varepsilon < \epsilon$. We choose $2t+1$ large enough so that 
\begin{align*}
   \frac{(2t+1)^d - (2t+1 - 2\delta)^d}{(2t+1)^d} &< \varepsilon \mbox{, and}\\
   \frac{1-\epsilon}{\#R_t} \sum_{\vec{j} \in R_t } \Phi \left( f (\vec{j}) \right) &\leq D\left( \Phi( f) \right).
\end{align*}

By the first condition, then, we have that
\begin{equation*}
   \frac{\#\left\{ \vec{j} \in R_{t, \delta} : \max_{N \in \Lambda} \frac{1}{\# S_N} \sum_{\vec{n} \in S_N} f\left( \vec{j} + \vec{n} \right) \geq K \right\}}{ \#R_t} \geq 1-2\epsilon.
\end{equation*}

\Pa
Let $\left(X, \mathcal{B}, m, T \right)$ be a probability space with an associated free $\Z^d$-action, $T$. We form a tower complex with the dimensions $R_t$- that is, with $2t+1$ levels in each coordinate. We denote each level by $E_{\vec{m}} = T_{\vec{m}} B$, where $B$ is the base of the towers, and the error set by $E_r$, with measure less than $\varepsilon$. 

\Pa
Define $\bar{f}: X \rightarrow \R$ by
\begin{equation*}
   \bar{f}(x) = \sum_{\vec{i} \in R_{t,\delta}} f(\vec{i}) \ind_{E_{\vec{i}}}(x).
\end{equation*}

We then have that 
\begin{align*}
\int_X \Phi \left( \bar{f} \right) \,\, dm &= \sum_{\vec{i} \in R_t} \int_X \Phi \left( \bar{f} \right) \ind_{E_{\vec{i}}} \,\, dm +\int_X \Phi \left( \bar{f} \right) \ind_{E_r} \,\, dm \\
&= \sum_{\vec{i} \in R_t} \int_X \Phi \left( f(\vec{i}) \right) \ind_{E_{\vec{i}}} \,\, dm\\
&=  \sum_{\vec{i} \in R_t}  \Phi \left( f(\vec{i}) \right) mE_{\vec{i}}\\
&< \frac{1 - \varepsilon}{\# R_t} \sum_{\vec{i} \in R_t}  \Phi \left( f(\vec{i}) \right)\\
&< D( \Phi \left( f\right) \leq 1.
\end{align*}

\Pa
So $\int_X \Phi \left( \bar{f} \right) \,\, dm \leq 1$ and $\bar{f} \in \Phi(L)$.

\Pa 
Now, if $x \in E_{\vec{i}}$, then $\bar{f} \left(T_{\vec{n}} x \right) = f (\vec{i} + \vec{n})$ and we have that
\begin{equation*}
   \max_{N \in \Lambda} \frac{1}{\#S_N} \sum_{\vec{n} \in S_N} \bar{f} \left( T_{\vec{n}} x \right) \geq K,
\end{equation*}
so long as 
\begin{equation*}
   \vec{i} \in \left\{ \vec{j} \in R_{t, \delta} : \max_{N \in \Lambda} \frac{1}{\# S_N} \sum_{\vec{n} \in S_N} f(\vec{j} + \vec{n}) \geq K \right\}.
\end{equation*}

\Pa
So we have that 
\begin{align*}
  &m\left\{ x \in E_{\vec{i}} :  \max_{N \in \Lambda}\frac{1}{\# S_N} \sum_{\vec{n} \in S_N}  \bar{f} \left(T_{\vec{n}} x \right) \geq K \right\}\\ 
 &\geq \# \left\{ \vec{j} \in R_{t, \delta} :  \max_{N \in \Lambda} \frac{1}{\# S_N} \sum_{\vec{n} \in S_N}  f(\vec{j} + \vec{n}) \geq K \right\} mE_{\vec{i}}.
\end{align*}

Taking the sum over all the $E_{\vec{i}}$, 
\begin{equation*}
   m\left\{ x \in X :  \max_{N \in \Lambda}\frac{1}{\# S_N} \sum_{\vec{n} \in S_N}  \bar{f} \left(T_{\vec{n}} x \right) \geq K \right\} \geq (1-2 \epsilon) (1- \varepsilon) > 1- 3 \epsilon.
\end{equation*}

Since $\Phi$ takes on arbitratily small values near $0$, and since $\Phi(0) =0$, for every $\bar{f}$ with $$\int_X \Phi \left( \bar{f} \right) \,\, dm \leq 1$$ and every positive integer $\alpha$, there is a number $M_\alpha$ so that
\begin{equation*}
   \int_X \Phi \left( \frac{|\bar{f}|}{M_\alpha} \right) \,\, dm \leq 2^{-\alpha}.
\end{equation*}

\Pa
Fix $M_\alpha$ and set $K = \alpha M_\alpha$, $\epsilon = \frac{1}{3\alpha}$, and $\bar{g}_\alpha = \frac{|\bar{f}|}{M_\alpha}$. We then have 
\begin{equation*}
   m\left\{x : \max_{N \in \Lambda} \frac{1}{\# S_N} \sum_{\vec{n} \in S_N}   \bar{g}_\alpha \left( T_{\vec{n}} x \right) \geq \alpha \right\} > 1 - \frac{1}{\alpha}.
\end{equation*}
If $\bar{g} = \sup_\alpha \bar{g}_\alpha$, then
\begin{align*}
   \int_X \Phi \left(  \bar{g} \right) \,\, dm &= \int_X \sup_\alpha \Phi \left(  \bar{g}_\alpha \right) \,\, dm\\
	&\leq \int_X \sum_{\alpha = 1}^\infty \Phi \left(  \bar{g}_\alpha \right)  \,\, dm \\
	&= \sum_{\alpha = 1}^\infty \int_X \Phi \left(  \bar{g}_\alpha \right)  \,\, dm \\
	&\leq 1
\end{align*}
So we have that $\bar{g} \in \Phi \left( L\right)$.

\Pa
Let 
\begin{equation*}
   A_\alpha = \left\{ x : \sup_N \frac{1}{\#S_N} \sum_{\vec{n} \in S_N } \bar{g}\left( T_{\vec{n}} x \right) \geq \alpha\right\},
\end{equation*}
and set 
\begin{equation*}
   A = \cap_{m=1}^\infty \cup_{\alpha = m}^\infty A_\alpha.
\end{equation*}

\Pa
We then have that $mA = 1$ and if $x \in A$, 
\begin{equation*}
   \sup_N \frac{1}{\#S_N} \sum \bar{g}\left( T_{\vec{n}} x\right) = \infty.
\end{equation*}

\end{proof}

\Pa
\begin{rem}
   It is interesting that the cubes which we use here are, in one sense, arbitrary: any F\o lner monotile in $\Z^d$ which is comparable to $B_N$ should serve.  Indeed, if we were to redefine what we mean by sparse and $\infty$-sweeping out, we would not even need them comparable to the balls- only that the monotile sequece is of, at most, polynomial growth.
\end{rem}

\section{Proofs of Theorems \ref{One} and \ref{Two}}
For the proofs of the theorems, we rely upon an adaptation of the perturbation method first developed in \cite{Wierdl1998}.

\Pa
\begin{proof}[Proof of Theorem \ref{One}:]
   Let $1 \leq q < \infty$, $D$ be sparse in $\Z^d$, and suppose $\left\{ D_N \right\}$ is a universally $L^p$-good sequence of sets for all $p>q$. Our goal is to construct a perturbation of $D$, $S$, that forces averages taken over the sequence $\left\{ S_N \right\}$ to be diverge a.e. for some function, $f \in L^q$. The poor averaging behavior of such a set will be caused by introducing elements that impart a regular structure; however, this structure will be tied only to sets over which the averages are taken, not the elements themselves. 

\Pa
   In order that we have enough room to add new elements, we require that $D$ is of zero density. We note that since the cube of side length $2N+1$ is both contained in, and grows at the same rate (up to a constant) as, the ball of radius $N$, we have that
\begin{equation*}
   \lim_{N \to 0} \frac{\# \left( D \cap B_N \right)}{\# B_N} \rightarrow 0
\end{equation*}
   if and only if
\begin{equation*}
   \lim_{N \to 0} \frac{\# \left( D \cap R_N \right)}{\# R_N} \rightarrow 0.
\end{equation*}

\Pa
   We now define the shell
\begin{equation*}
   I_n = \left\{ a \in \Z^d : |a_i| = n, \,\, 1 \leq i \leq d \right\}.
\end{equation*}
   In order to create our perturbation, we will add a specific number of elements from specially chosen shells, but we will not otherwise specify the elements.\\
\Pa
   Let 
\begin{equation*}
   A_u = \left\{ k : k = 2^u, 2^u+1, ..., 2^{u+1}-1 \right\}.
\end{equation*}
We'll use $A_u$ to provide the additional structure required in our perturbation. We will do this by adding elements from shells whose indicies are congruent to $k$ mod $2^u$ if $k \in A_u$. In order to force divergence for $L^q$ while at the same time preserving the convergence in the $L^p$ spaces for $p>q$, however, we will need to select intervals of indicies that have few elements of $D$, that are large enough to contain the number of elements we wish to add, that are of the right size to ensure that the resulting set $S$ is a perturbation of $D$ (with respect to the sequence $\left\{R_N \right\}$), and, finally, don't overlap. More specifically, if $[n_k, 2n_k)$ are our intervals, we want
\begin{align}
   &n_k > 2n_{k-1},\label{overlap}\\
   &u2^u\#D_{n_k} < n_k \mbox{ for } k \in A_u, \label{bigenough}\\
   &\# D_{2n_k} \leq 3 \#D_{n_k}, \mbox{ and} \label{sparse}\\
   &\sum_{i=1}^{k-1} \# D_{n_i}  < \left( \frac{u}{2^u} \right)^{1/q} \# D_{n_k}, \mbox{ for } k \in A_u. \label{perturbation}
\end{align}

\Pa
   Our first requirement makes sure that these intervals do not overlap. For our second, suppose $u$ is fixed. Then from this range of $n_k$ to $2n_k$ we would like to choose $\left( \frac{u}{2^u} \right)^{1/q} \#D_{n_K}$ elements from shells $I_j$ with $j \equiv k$ mod $2^u$ to add to $D$. If $n_k > \left( \frac{u}{2^u} \right)^{1/q} \#D_{n_K} \cdot 2^u$, we can guarantee that there are enough elements. The third requirement ensures that there are not too many elements of $D$ in these shells, and the fourth will help make $S$ a perturbation. We must now show that there is a sequence $\left\{ n_k \right\}$ which meets our requirements. 

\Pa 
   Here we rely on the sparseness of $D$. There is a sequence $\left\{ m_j \right\}$ with the properties
\begin{align*}
   &\lim_{j \to \infty} \frac{\# D_{m_j}}{ \# R_{m_j}} =0, \mbox{ and}\\
   &   \frac{\# D_{m_j}}{ \# R_{m_j}} \leq  \frac{\# D_{m}}{ \# R_{m}}, \mbox{ for } m \leq m_j.
\end{align*}
   We will construct the elements of $\{ n_k \}$ using elements this sequence, letting $n_k = \lfloor m_j/2 \rfloor$ for large enough $j$.

\Pa
   We can choose the elements of $\{ n_k \}$ to satisfy both condition (\ref{overlap}) and condition (\ref{perturbation}). For condition (\ref{bigenough}), we note that
\begin{equation*}
   \frac{\# D_{n_k}}{ \# R_{n_k}} = \frac{\# D_{\lfloor m_j/2 \rfloor}}{ \# R_{\lfloor m_j/2 \rfloor}} < 3^d \frac{\# D_{ m_j }}{ \# R_{ m_j}}.
\end{equation*}
   Thus the sequence $\# D_{n_k} / \# R_{n_k} \to 0$, allowing us to satisfy (\ref{bigenough}) by choosing $n_k$ large enough.

\Pa
   This leaves us with condition (\ref{perturbation}). Here again we take advantage of our definition of $\{n_k \}$ and the sparseness of $D$.
\begin{align*}
   \# D_{2n_k} &\leq \# D_{m_j} = \frac{\# D_{ m_j }}{ \# R_{ m_j}} \#R_{ m_j} \\
	       &\leq \frac{\# D_{\lfloor m_j/2 \rfloor}}{ \# R_{\lfloor m_j/2 \rfloor}} \#R_{ m_j} \\
               &< 3^d \# D_{\lfloor m_j/2 \rfloor} = 3^d \# D_{n_k}.
\end{align*}

\Pa 
\begin{rem}RestarJuly5
   It is interesting to note here that both of these last two statements were consequences of the polynomial growth of $\Z^d$: the first directly, the second through our use of the fact that the volume of cubes of different diameters has a constant ratio.
\end{rem}

\Pa 
Having shown that we have room enough to do so, the perturbed set $S$ is constructed in the following manner. Consider the family of shells
\begin{equation*}
   \mathcal{I}_{n_k} = \left\{ I_m : n_k \leq m \leq 2n_k, m \equiv k \mbox{ mod} 2^u \mbox{, if } k \in A_u \right\}.
\end{equation*}
From this family, we choose, \textit{as we please}, $\left( u/ 2^u \right)^{1/q} \# D_{n_k}$ elements; we call the set consisting of these elements $E_k$. We then have
\begin{equation*}
   S= D \cup \left( \cup_{k}^\infty E_k\right).
\end{equation*}
Since $S$ is formed by adding elements to $D$, in order to show that $S$ is a perturbation we need only show that
\begin{equation*}
   \frac{\# \left(S_N \setminus D_N\right)}{\# D_N} \to 0.
\end{equation*}

\Pa
Suppose that $n_k \leq N < n_{k+1}$. Then
\begin{align*}
   \# \left(S_N \setminus D_N\right) &< \#E_k + \sum_{i=1}^{k-1} \# D_{n_i} \\
			       &\leq 2 \left( \frac{u}{2^u} \right)^{1/q} \#D_{n_k} \\
			       &\leq 2 \left( \frac{u}{2^u} \right)^{1/q} \#D_N.
\end{align*}
So, since as $k$ grows, so must $u$, we have 
 \begin{equation*}
    \lim_{N \to \infty} \frac{\# \left(S_N \setminus D_N \right)}{\# D_N} = \lim_{u \to \infty} 2 \left( \frac{u}{2^u} \right)^{1/q} =0.
 \end{equation*}

\Pa
We now wish to show that $\left\{ S_N \right\}$ is universally good for $L^p$, $p >q$. Fix $p$ and a measure-preserving $\Z^d$-action, $\left(X, \B, m, T \right)$. Since 
\begin{align*}
   \frac{1}{\# S_N} \sum_{g \in S_N} f\left(T_g x\right) &=  \frac{1}{\# D_N} \sum_{g \in S_N} f\left(T_g x\right)\\
							 &=  \frac{1}{\# D_N} \sum_{g \in D_N} f\left(T_g x\right) +  \frac{1}{\# D_N} \sum_{g \in S_N \setminus D_N} f\left(T_g x\right), 
\end{align*}
 and the first average converges, we need only show that 
\begin{equation}\label{limsup}
   \limsup \frac{1}{\# D_N} \sum_{g \in S_N \setminus D_N} f\left(T_g x\right) =0.
\end{equation}

Since, for $n_k \leq N < 2n_k$, we have 
\begin{equation*}
    \frac{1}{\# D_N} \sum_{g \in S_N \setminus D_N} f\left(T_g x\right) \leq  \frac{1}{\# D_{n_k}} \sum_{g \in S_{2n_k} \setminus D_{2n_k}} f\left(T_g x\right), 
\end{equation*}
we limit our attention to this second average; (\ref{limsup}) will follow if we can show
\begin{equation*}
   \int_X \sum_{k=1}^\infty \left( \frac{1}{\# D_{n_k}} \sum_{g \in S_{2n_k} \setminus D_{2n_k}} f\left(T_g x\right) \right)^p dm < \infty.
\end{equation*}

\Pa
Applying Tonelli and the triangle inequality, we have
\begin{align*}
    \int_X \sum_{k=1}^\infty \left( \frac{1}{\# D_{n_k}} \sum_{g \in S_{2n_k} \setminus D_{2n_k}} f\left(T_g x\right) \right)^p dm &\leq  \sum_{k=1}^\infty  \left\| \frac{1}{\# D_{n_k}} \sum_{g \in S_{2n_k} \setminus D_{2n_k}} f\left(T_g x\right) \right\|_{L^p}^p\\
	&\leq \left\| f \right\|_{L^p}^p \sum_{k=1}^\infty \left( \frac{\# \left( S_{2n_k} \setminus D_{2n_k} \right)}{\# D_{n_k}} \right)^p\\
        &= \left\| f \right\|_{L^p}^p \sum_{u=1}^\infty \sum_{k \in A_u} \left( \frac{\# \left( S_{2n_k} \setminus D_{2n_k} \right)}{\# D_{n_k}} \right)^p\\
        &= \left\| f \right\|_{L^p}^p \sum_{u=1}^\infty 2^u \left( \frac{1}{\# D_{n_k}} \cdot 3^d \# D_{n_k} \left( \frac{u}{2^u} \right)^{1/q}  \right)^p\\
	&= \left\| f \right\|_{L^p}^p \sum_{u=1}^\infty  3^{dp}  \frac{u^{p/q}}{2^{u(p/q-1)}} \\
	&= C_p \left\| f \right\|_{L^p}^p < \infty.
\end{align*}

\Pa
All that remains is to show that our perturbation does not allow for convergence for $f \in L^q$. To do this, we make use of Lemma \ref{Rohlin}. Our goal is then to find a reasonably well-behaved function $f: \Z^d \to \R$ for which 
\begin{equation*}
   \# \left\{ j : \max_{N \in \Lambda} \frac{1}{\# S_N} \sum_{g \in S_N} \left| f(j+g) \right| > K \right\} \geq (1- \varepsilon) \# R_N, \mbox{ i.o.}.
\end{equation*}
We'll meet this condition if we can show 
\begin{equation*}
   \left\{ j : \max_{k \in A_u} \frac{1}{\# S_{2n_k}} \sum_{g \in S_{2n_k}} \left| f(j+g) \right| > K(u) \right\} = \Z^d,
\end{equation*}
where $K$ is now an unbounded increasing function in $u$. 

\Pa 
Let 
\begin{equation*}
   E_{k,j} = \left\{ n \in E_k : n_j \equiv k \mbox{ mod } 2^u \right\}.
\end{equation*}
Then there is a $j$, $1 \leq j \leq d$, with 
\begin{equation*}
   \#E_{k,j} > \frac{1}{d} \left( \frac{u}{2^u} \right)^{1/q} \#D_{n_k}.
\end{equation*}
For each $k \in A_u$, we have a $j_k$; however, since there are $2^u$ such $k$, and only $d$ $j$'s, at least $\frac{2^u}{d}$ $k$'s must share a single $j$.  That is, there is a set $H \subset A_u$, $\#H \geq 2^u/d$, and a $j$, $1 \leq j \leq d$, so that 
\begin{equation*}
      \#E_{h,j} > \frac{1}{d} \left( \frac{u}{2^u} \right)^{1/q} \#D_{n_h}
\end{equation*}
for any $h \in H$. 

\Pa
Now, consider the set 
\begin{equation*}
   R_L(H, j) = \left\{ x \in R_L : x_j \equiv -h \mbox{ mod } 2^u, \mbox{ for some } h \in H \right\}.
\end{equation*}
This would be the set on which we'd like to build our divergence, using a suitable function $f$. But this set only covers $1/d$ of $R/L$. There are, however, injective maps $\pi_i: \Z/2^u \to \Z/2^u, 0 \leq i < d$, with $\pi_0$ denoting the identity map, and sets
\begin{equation*}
 R_{L,i} = \left\{ x \in R_L : \pi_i (x_j + h) = 0, \mbox{ for some } h \in H \right\}
\end{equation*}
so that 
\begin{equation*}
 R_L = \cup_{0 \leq i < d} R_{L, i}.
\end{equation*}

\Pa
Define $\phi$ by
\begin{equation*}
   \phi(x)= \left\{ 
\begin{array}{l l}
2^{u/q} & \quad \mbox{if $2^u | x_i$ for } 1 \leq i \leq d,\\
      0 & \quad \mbox{otherwise.}\\
\end{array} \right.
\end{equation*}
We then have that 
\begin{align*}
   D \left( |\phi|^q \right) &= \limsup_{L \to \infty} \frac{1}{\left( 2L+1 \right)^d} \sum_{x \in R_L} \left|f(x) \right|^q \\
                        &= \limsup_{L \to \infty} \frac{2^u}{\left( 2L+1 \right)^d} \# \left\{ x \in R_L: 2^u|x_i, 1\leq i \leq d \right\} \leq 1.
\end{align*}

\Pa
Now define $\phi_i$ by
\begin{equation*}
 \phi_i (x) = \phi\left(\pi_i(x)\right) \ind_{\left\{supp \left(\sum_{0 \leq k < i} \phi_k \right) \right\}^c}(x),
\end{equation*}
and $f$ by
\begin{equation*}
 f(x) = \frac{1}{d} \sum_{0 \leq i <d} \phi_i (x).
\end{equation*}
Since $D \left( |\phi|^q \right) \leq 1$, we likewise have
\begin{equation*}
 D \left( |f|^q \right) \leq 1,
\end{equation*}
and, for any $x \in \cup_{0 \leq i < d} R_{L,i}$ and $h \in H$, we have 
\begin{equation*}
  f(x+h) = \frac{ 2^{u/q}}{d}
\end{equation*}
Thus, for any $x \in \cup_{0 \leq i < d} R_{L,i}$ and $k \in H$,
\begin{align*}
    \frac{1}{\# S_{2n_k}} \sum_{g \in S_{2n_k}} \left| f(j+g) \right| &\geq  \frac{1}{3^d \#D_{n_k}} \sum_{g \in E_k} \left| f(j+g) \right| \\
     &\geq \frac{1}{3^d \#D_{n_k}} \sum_{g \in E_{k,j}} \frac{2^{u/q}}{d} \\
     &= 2^{u/q} \frac{1}{3^d \#D_{n_k}} \left( \frac{u}{2^u} \right)^{1/q} \frac{\#D_{n_k}}{d^2} \\
     &= \frac{u^{1/q}}{d^2 3^d}.
\end{align*}
So we have
\begin{equation*}
 \cup_{0 \leq i < d} R_{L,i} \subseteq \left\{ x \in R_L: \max_{k \in A_u} \frac{1}{\# S_{2n_k}} \sum_{g \in S_{2n_k}} \left| f(x+g) \right| > \frac{u^{1/q}}{d3^d} \right\},
\end{equation*}
and
\begin{equation*}
   \# \left\{ x \in R_L: \max_{k \in A_u} \frac{1}{\# S_{2n_k}} \sum_{g \in S_{2n_k}} \left| f(x+g) \right| > \frac{u^{1/q}}{d^2 3^d} \right\} = \# R_L.
\end{equation*}

\end{proof}

\begin{proof}[Proof of Theorem \ref{Two}:]
The proof of our second theorem proceeds in much the same vein as the first. Fix $q$ and let $1 \leq p<q$. Having defined the set $A_u$ precisely as before, and having established the existence of a suitable sequence $\left\{ n_k \right\}$ (which may also be identical to our previous selection), we select from each $\mathcal{I}_{n_k}$
\begin{equation*}
   \left( \frac{1}{u^2 2^u} \right)^{1/q} \#D_{n_k}.
\end{equation*}
Adding these to our sequence of sets $\left\{D_N \right\}$, we form the perturbation $\left\{S_N \right\}$.

\Pa
We now wish to show that the petrubation $\left\{S_N \right\}$ remains $L^q$-good. Fix a measure-preserving $\Z^d$-action, $\left(X, \B, m, T \right)$. Since 

\begin{align*}
   \frac{1}{\# S_N} \sum_{g \in S_N} f\left(T_g x\right) &=  \frac{1}{\# D_N} \sum_{g \in S_N} f\left(T_g x\right)\\
							 &=  \frac{1}{\# D_N} \sum_{g \in D_N} f\left(T_g x\right) +  \frac{1}{\# D_N} \sum_{g \in S_N \setminus D_N} f\left(T_g x\right), 
\end{align*}
 and the first average converges, we need only show that 
\begin{equation}\label{limsup}
   \limsup \frac{1}{\# D_N} \sum_{g \in S_N \setminus D_N} f\left(T_g x\right) =0.
\end{equation}

Since, for $n_k \leq N < 2n_k$, we have 
\begin{equation*}
    \frac{1}{\# D_N} \sum_{g \in S_N \setminus D_N} f\left(T_g x\right) \leq  \frac{1}{\# D_{n_k}} \sum_{g \in S_{2n_k} \setminus D_{2n_k}} f\left(T_g x\right), 
\end{equation*}
we limit our attention to this second average; as in the first proof, convergence will follow if we can show
\begin{equation*}
   \int_X \sum_{k=1}^\infty \left( \frac{1}{\# D_{n_k}} \sum_{g \in S_{2n_k} \setminus D_{2n_k}} f\left(T_g x\right) \right)^q dm < \infty.
\end{equation*}

\Pa
Applying Tonelli and the triangle inequality, we have
\begin{align*}
    \int_X \sum_{k=1}^\infty \left( \frac{1}{\# D_{n_k}} \sum_{g \in S_{2n_k} \setminus D_{2n_k}} f\left(T_g x\right) \right)^q dm 
	&\leq  \sum_{k=1}^\infty  \left\| \frac{1}{\# D_{n_k}} \sum_{g \in S_{2n_k} \setminus D_{2n_k}} f\left(T_g x\right) \right\|_{L^q}^q\\
	&\leq \left\| f \right\|_{L^q}^q \sum_{k=1}^\infty \left( \frac{\# \left( S_{2n_k} \setminus D_{2n_k} \right)}{\# D_{n_k}} \right)^q\\
        &= \left\| f \right\|_{L^q}^q \sum_{u=1}^\infty \sum_{k \in A_u} \left( \frac{\# \left( S_{2n_k} \setminus D_{2n_k} \right)}{\# D_{n_k}} \right)^p\\
        &= \left\| f \right\|_{L^q}^q \sum_{u=1}^\infty 2^u \left( \frac{1}{\# D_{n_k}} \cdot 3^d \# D_{n_k} \left( \frac{1}{u^2 2^u} \right)^{1/q}  \right)^q\\
	&= \left\| f \right\|_{L^p}^p \sum_{u=1}^\infty  3^{dq}  \frac{1}{u^2} \\
	&= C_q \left\| f \right\|_{L^q}^q < \infty.
\end{align*}

\Pa
We now wish to show that our added elements are enough to cause divergence for $p<q$. Once again, parallel to our earlier argument, we make use of Lemma \ref{Rohlin}. 

\Pa
Define $f$ by
\begin{equation*}
   f(x)= \left\{ 
\begin{array}{l l}
2^{u/p} & \quad \mbox{if $2^u | x_i$ for } 1 \leq i \leq d,\\
      0 & \quad \mbox{otherwise.}\\
\end{array} \right.
\end{equation*}
We then have that 
\begin{align*}
   D \left( |f|^p \right) &= \limsup_{L \to \infty} \frac{1}{\left( 2L+1 \right)^d} \sum_{x \in R_L} \left|f(x) \right|^p \\
                        &= \limsup_{L \to \infty} \frac{2^u}{\left( 2L+1 \right)^d} \# \left\{ x \in R_L: 2^u|x_i, 1\leq i \leq d \right\} \leq 1.
\end{align*}

\Pa 
Let 
\begin{equation*}
   E_{k,j} = \left\{ n \in E_k : n_j \equiv k \mbox{ mod } 2^u \right\}.
\end{equation*}
Then there is a $j$, $1 \leq j \leq d$, with 
\begin{equation*}
   \#E_{k,j} > \frac{1}{d} \left( \frac{1}{u^2 2^u} \right)^{1/q} \#D_{n_k}.
\end{equation*}
For each $k \in A_u$, we have a $j_k$; however, since there are $2^u$ such $k$, and only $d$ $j$'s, at least $\frac{2^u}{d}$ $k$'s must share a single $j$.  That is, there is a set $H \subset A_u$, $\#H \geq 2^u/d$, and a $j$, $1 \leq j \leq d$, so that 
\begin{equation*}
      \#E_{h,j} > \frac{1}{d} \left( \frac{1}{u^2 2^u} \right)^{1/q} \#D_{n_h}
\end{equation*}
for any $h \in H$. 

\Pa
Constructing $R_{L,i}$ and $f$ as before, we have for any $x \in \cup_{0 \leq i <d} R_{L,i}$ and $k \in H$,
\begin{align*}
    \frac{1}{\# S_{2n_k}} \sum_{g \in S_{2n_k}} \left| f(j+g) \right| &\geq  \frac{1}{3^d \#D_{n_k}} \sum_{g \in E_k} \left| f(j+g) \right| \\
     &\geq \frac{1}{d 3^d \#D_{n_k}} \sum_{g \in E_{k,j}} 2^{u/q} \\
     &= 2^{u/p} \frac{1}{d 3^d \#D_{n_k}} \left( \frac{1}{u^2 2^u} \right)^{1/q} \frac{\#D_{n_k}}{d} \\
     &= \frac{2^{\gamma u}}{u^{2/q}d^2 3^d}
\end{align*}
where $\gamma = \frac{q-p}{pq} > 0$. So 
\begin{equation*}
  \cup_{0 \leq i <d} R_{L,i} \subseteq \left\{ x \in R_L: \max_{k \in A_u} \frac{1}{\# S_{2n_k}} \sum_{g \in S_{2n_k}} \left| f(x+g) \right| > \frac{2^{\gamma u}}{u^{2/q}d3^d} \right\},
\end{equation*}
and, since $\# R_L(H,j) \geq \#R_L /d$,
\begin{equation*}
   \# \left\{ x \in R_L: \max_{k \in A_u} \frac{1}{\# S_{2n_k}} \sum_{g \in S_{2n_k}} \left| f(x+g) \right| > \frac{2^{\gamma u}}{u^{2/q}d^2 3^d} \right\} = \# R_L.
\end{equation*}

\end{proof}

\section{Concluding Remarks}
For Corollary \ref{Balls} we need only note that $S(N)$ is a perturbation of $D(N)$. The proof follows from conditions (\ref{comparable}) and (\ref{growth}) and the fact that the ratio of the measures of elements of a perturbed sequence and corresponding elements of its parent must go to 1. The need for both conditions in the corollary, however, relates to one of the central problems in extending this result to a more general group setting. Without the conditions, we need not have a relationship between  $\# S(N)$ and $\# S_N$. However, if our set $D$ is, in some sense, evenly distributed (for example, the randomly generated sets of \cite{LPR}), we can count on the sets $S_N$ to not have too much weight in the corners.  

\Pa 
In addition to the problem of insuring a relationship between different ways of averaging, there are several obstacles to proving a lemma analagous to \ref{Rohlin} in the more general group setting. The particular problem is that in the more general case, we do not have a F\o lner monotile that is comparable to the nested sequence of balls by which density is defined. In $\Z^d$, any nested sequence of rectangles fills this role. While the existence of a sequence of monotiles for any solvable group is shown in \cite{Weiss2001}, establishing the existence of such a sequence comparable in measure to the balls remains. Completing this step would likely lead to a comparable perturbation result for virtually nilpotent groups; however, this approach is unlikely to work for groups with superpolynomial rates of growth.

\bibliography{PerturbationSparseErgodicAverages.bib}
\bibliographystyle{plain}

\end{document}